\documentclass[10pt]{article} 
\usepackage{amsmath,amssymb,latexsym,amsthm,enumerate,amscd,hyperref,enumitem} \usepackage[abbrev]{amsrefs} 
\usepackage{graphicx}
\usepackage{pinlabel}
\usepackage{graphicx,tikz, tikz-cd}

\newtheoremstyle{dotless}{}{}{\itshape}{}{\bfseries}{}{ }{}

\newtheorem{Theorem}{Theorem}[section]

\newtheorem{Question}[Theorem]{Question} 
\newtheorem{lemma}[Theorem]{Lemma}

\newtheorem*{ConeLemma}{Cone Lemma} 
\newtheorem*{JoinLemma}{Join Lemma} 

\theoremstyle{definition} 
\newtheorem{definition}[Theorem]{Definition} 
\newtheorem{remark}[Theorem]{Remark}

\newtheorem*{Acknowledgements}{Acknowledgements}

\newtheorem*{Example}{Example}

\DeclareMathOperator{\piprod}{\raisebox{-0.1em}{\huge{$\pi$}}\kern -0.2em}

\newcommand{\cac}{\mathcal {C}}

\newcommand{\rr}{{\mathbb R}}

\newcommand{\zz}{{\mathbb Z}}

\newcommand{\vkt}{vk_{\zz/2}}

\newcommand{\gdim}{\operatorname{gd}}
\newcommand{\obdim}{\operatorname{obdim}}
\newcommand{\eqobdim}{\operatorname{eqobdim}}
\newcommand{\actdim}{\operatorname{actdim}}
\newcommand{\updim}{\operatorname{updim}}

\newcommand{\Cone}{\operatorname{Cone}}
\newcommand{\Conf}{\operatorname{Conf}}
\newcommand{\Wu}{\operatorname{Wu}}
\newcommand{\Supp}{\operatorname{Supp}}
\newcommand{\CWu}{\operatorname{CWu}}
\newcommand{\vk}{\operatorname{vk}}
\newcommand{\vktwo}{\vk_{\zz/2}}

\def\clap#1{\hbox to 0pt{\hss#1\hss}}

\newcommand{\comment}[1]{} 
 
\newcommand{\ga}{\alpha}

\newcommand{\geps}{\varepsilon}

\newcommand{\gs}{\sigma} 
 
\newcommand{\gt}{\tau}

\newcommand{\gD}{\Delta}

\newcommand{\edge}{\operatorname{Edge}}

\newcommand{\Hom}{\operatorname{Hom}} 
\newcommand{\Homeo}{\operatorname{Homeo}}

	{ 
\end{enumerate}
} 

	{ 
\end{enumerate}
} 

\newenvironment{enumeratei'}{ 
\begin{enumerate}[\upshape (i)$'$]}
	{ 
\end{enumerate}
} 

\newenvironment{enumerate1'}{ 
\begin{enumerate}[\upshape (1)$'$]}
	{ 
\end{enumerate}
} 

\newenvironment{enumeratea'}{ 
\begin{enumerate}[\upshape (a)$'$]}{ 
\end{enumerate}
}

\usepackage{color}
\usepackage[outerbars,color]{changebar}
\usepackage{ifthen}

\newboolean{usecolor}
\setboolean{usecolor}{true}

\ifthenelse{\boolean{usecolor}}
{
  \definecolor{colore}{cmyk}{0,1,0.6,0}
  \definecolor{coloregen}{cmyk}{0.7,0,1,0}
  \definecolor{coloresimo}{cmyk}{1,0.6,0,0}
}
{
  \definecolor{colore}{cmyk}{0,0,0,1}
  \definecolor{coloregen}{cmyk}{0,0,0,1}
  \definecolor{coloresimo}{cmyk}{0,0,0,1}
}

\numberwithin{equation}{section} 
\title{Properly discontinuous actions versus uniform embeddings}
\author{Kevin Schreve}

\begin{document}
\maketitle

\begin{abstract} 

\noindent Whenever a finitely generated group $G$ acts properly discontinuously by isometries on a metric space $X$, there is an induced uniform embedding (a Lipschitz and uniformly proper map) $\rho: G \rightarrow X$ given by mapping $G$ to an orbit. 
We study when there is a difference between a finitely generated group $G$ acting properly on a contractible $n$-manifold and uniformly embedding into a contractible $n$-manifold. For example, Kapovich and Kleiner showed that there are torsion-free hyperbolic groups that uniformly embed into a contractible $3$-manifold but only virtually act on a contractible $3$-manifold. We show that $k$-fold products of these examples do not act on a contractible $3k$-manifold.

\medskip 

\noindent
\textbf{AMS classification numbers}. Primary: 20F36, 20F55, 20F65, 57S30, 57Q35, 
Secondary: 20J06, 32S22
\smallskip

 \end{abstract}

\section{Introduction}\label{s:introduction}

For a finitely generated group $G$, the \emph{action dimension} of $G$, denoted $\actdim(G)$, is the minimal dimension of contractible manifold $M$ that admits a properly discontinuous $G$-action. If $G$ is torsion-free, then the quotient $M/G$ is a manifold model for the classifying space $BG$, so the action dimension is precisely the minimal dimension of such a model. The \emph{geometric dimension} is the minimal dimension of a CW-model for $BG$. 

Given a properly discontinuous action of $G$ on $M$, and given any choice of basepoint $m_0$, there is an orbit map $\rho:G \rightarrow M$ defined by $g \rightarrow g.m_0$. After choosing a proper $G$-invariant metric on $M$, this map is Lipschitz and uniformly proper; we call such a map   a \emph{uniform embedding}.  Furthermore, the manifold $M$ is uniformly contractible around the image of $G$ (see Section $2$ for precise definitions).
The \emph{uniformly proper dimension} of $G$, denoted $\updim(G)$, is the minimal dimension of contractible manifold $M$, equipped with a proper metric, so that there is a uniform embedding $\rho: G \rightarrow M$ so that $M$ is uniformly contractible around the image of $G$. The orbit map of a properly discontinuous action shows that $$\updim(G) \le \actdim(G)$$ 

We now review some of the known relations between $\updim(G)$ and $\actdim(G)$. Somewhat surprisingly,  the two dimensions coincide for most of the examples where they have been computed. Such groups include lattices in Lie groups \cite{bf}, mapping class groups \cite{despotovic}, many Artin groups \cite{ados} \cite{dh}, and torsion-free lattices in Euclidean buildings \cite{ks}. These results all come from computing a lower bound to $\updim(G)$, called the \emph{obstructor dimension}, which was defined by Bestvina, Kapovich, and Kleiner \cite{bkk}. We will come to this later in the introduction. 

There are also examples where $\updim(G)$ is strictly less than $\actdim(G)$, but equal to $\actdim(H)$ for $H$ a finite index subgroup of $G$. These examples are relatively easy to construct when $G$ is allowed to have torsion, for example there are many virtually free groups which do not act properly on the plane (such as the free product of alternating groups $A_5 \ast A_5$). Torsion-free examples were constructed by Kapovich-Kleiner in \cite{kk} and Hruska-Stark-Tran in \cite{hst}. In both cases, the groups constructed were virtually $3$-manifold groups but not $3$-manifold groups. The constructions have a similar flavor, roughly one glues surfaces together along simple closed curves using degree $k$ covering maps for $k > 1$. 
In both cases, the obstruction to properly acting on a contractible $3$-manifold comes from analyzing the action of the group on collections of codimension-one hypersurfaces in the universal cover $EG$ of $BG$, and applying the coarse Jordan separation theorem of \cite{kk}. 

There are fewer known examples where $$\updim(G) <  \min_{[G:H] < \infty}\actdim(H).$$ In fact, the only common examples we know are the Baumslag-Solitar groups $$BS(m,n) = \langle x,y | xy^mx^{-1} = y^n\rangle$$ for $m \ne n$. These uniformly embed into a uniformly contractible $3$-manifold (which is a thickening of the Cayley 2-complex), but for a variety of reasons are not $3$-manifold groups if $m \ne n$ (and this is true for finite index subgroups as well). A theorem of Stallings \cite{stallings} implies that for groups with a finite $BG$, $\actdim(G)$ is bounded above by twice the geometric dimension of $G$. In particular, $\actdim= 4$ for the two previous examples (there are also obvious 4-dimensional manifold models of $BG$). 

Kapovich-Kleiner have higher-dimensional results in this direction; for example they show that the group $BS(m,n) \times \zz^k$ does not act properly on a uniformly contractible $(3+k)$-manifold. 
Note that we have the obvious inequalities $$\updim(\Gamma_1 \times \Gamma_2) \le \updim(\Gamma_1) + \updim(\Gamma_2)$$ and $$\actdim(\Gamma_1 \times \Gamma_2) \le \actdim(\Gamma_1) + \actdim(\Gamma_2).$$
It is still open if this last inequality is strict for products of the above examples.

\begin{Question}
Let $G$ be the $k$-fold direct product of the examples in \cite{kk} or \cite{hst}.  What is $\actdim(G)$? Same question for products of Baumslag-Solitar groups with $m \ne n$. 
\end{Question}

It follows from \cite{bkk} that the uniformly proper dimension of these is $ = 3k$, hence $3k \le \actdim(G) \le 4k$. The difficulty here is that all of the above computations rely on studying the action of $G$ on codimension-one hypersurfaces inside $EG$, and showing that this action in incompatible with a group acting on hypersurfaces in a contractible $3$-manifold.  After crossing with $\zz^n$ (or $\pi_1$ of any closed aspherical manifold), there are still codimension-one hypersurfaces in $EG$. However, the $k$-fold product of these examples now has codimension-$k$ hypersurfaces inside $EG$, and the same analysis doesn't apply. For such products, we have the following theorem, which handles some cases where the planes have larger codimension. 

\begin{Theorem}\label{thm:main}
Let $G$ be the $k$-fold direct product of the examples in \cite{kk} or \cite{hst}, where the degree of the covering map is a multiple of $4$. 
Then $$\actdim(G) \ge 3k+1.$$ \end{Theorem}

Similar results hold for some virtually free groups, see Section 6 for the precise statements. The conditions on the degree of the cover are an unfortunate fault of our method. Of course, these groups have finite index subgroups which have $\actdim(G) = 3k$.

Before describing the methods that go into the proof of Theorem \ref{thm:main}, let us recall the \emph{obstructor dimension} $\obdim(G)$ of a group $G$ as defined in \cite{bkk}. This is based on the  \emph{$\zz/2$-valued van Kampen obstruction} to embedding finite subcomplexes into $\rr^n$, which is an $n$-dimensional class, denoted $\vktwo^n$, in the cohomology of the unordered $2$-point configuration space $\cac(K)$ with $\zz/2$-coefficients, see subsection \ref{ss:vk} for details. 
A finite complex $K$ is an \emph{$n$-obstructor} if $\vktwo^n(K) \ne 0$; in particular this implies that $K$ does not embed into $\rr^n$. 

Roughly speaking, the \emph{obstructor dimension} of a finitely generated group $G$ is the maximal $n+2$ so that there is an $n$-obstructor $K$ and a uniformly proper embedding $f:K \times \rr^+ \rightarrow EG$. Bestvina, Kapovich and Kleiner show that $$\obdim(G) \le \updim(G).$$ 

The moral we follow is that if obstructor complexes give lower bounds for $\updim(G)$, then simplicial complexes with a group action which do not equivariantly embed into $\rr^n$ should give lower bounds for $\actdim(G)$. Again, we require a cohomological obstruction to equivariantly embedding the complex. In this case, we use an ambient isotopy invariant, called the \emph{Wu invariant},  of an embedding of $K$ into $\rr^{n+1}$. This invariant also lives in the $n^{th}$ degree cohomology of $\cac(K)$, though with twisted integral coefficients, and its image in $H^n(\Conf(K); \zz/2)$ upon reducing the coefficients to $\zz/2$ is precisely the $\zz/2$-valued van Kampen obstruction. 

If $K$ is a graph, then the Wu invariant has been often used to study embeddings of $K$ into $\rr^3$, see for example \cite{ffn} \cite{taniyama}. Flapan in \cite{f} also used the linking number of images of subgraphs to obstruct certain equivariant embeddings of the complete graph $K_n$ into $\rr^3$. 
If a finite group $H$ acts on $K$, and an embedding $f:K \rightarrow \rr^n$ is equivariant with respect to some representation $\rho: H \rightarrow \Homeo^+(\rr^n)$, then the $\Wu$ invariant of $f$ is fixed under the $H$-action on $H^n(\Conf(K))$ (since all orientation-preserving homeomorphisms of $\rr^n$ are isotopic to the identity). We roughly define an \emph{equivariant obstructor} to be a finite $H$-complex $K$ which is an $n$-obstructor and which does not admit an invariant $\Wu$ class. 

We then roughly define the \emph{equivariant obstructor dimension}, $\eqobdim(G)$,  of a finitely generated group $G$ to be the maximal $n+2$ so that there is an equivariant $n$-obstructor $K$ and a uniformly proper embedding $f:K \times \rr^+ \rightarrow EG$ which is coarsely $H$-equivariant, see Section \ref{s:coarsewu}.
It will follow from the definitions that $\obdim(G) \le \eqobdim(G) \le \obdim(G) + 1$. We will show that $$\eqobdim(G) \le \actdim(G)$$ and $$\eqobdim(G_1 \times G_2) \ge \eqobdim(G_1) + \eqobdim(G_2) - 1$$ which will imply Theorem \ref{thm:main}. 

This paper is structured as follows. If Section \ref{s:background}, we review some necessary background information. In Section \ref{s:equivariant} we define an equivariant obstructor complex and show that such a complex does not equivariantly embed into $\rr^n$. In Sections 4 and 5 we develop the coarse analogue of this. In Section \ref{s:examples}, we apply this to compute the action dimension of a number of examples. 

\begin{Acknowledgements}
I would like to thank Boris Okun and Shmuel Weinberger for helpful conversations. The author is partially supported by NSF grant DMS-1045119.
\end{Acknowledgements}

\section{Background}\label{s:background}

\subsection{Uniformly proper dimension and coarse topology}\label{ss:upd}

Recall that a metric space is \emph{proper} if closed metric balls are compact, and that a map between two spaces is proper if preimages of compact sets are compact. Two maps $f_0, f_1: X \rightarrow Y$ are \emph{properly homotopic} if there is a proper map $F: X \times [0,1] \rightarrow Y$ so that $F|_{X \times 0} = f_0$ and $F|_{X \times 1} = f_1$. 

 Let $X$ and $Y$ be two proper metric spaces. 
A proper map $f: X \rightarrow Y$ is \emph{uniformly proper} if there exists a proper function $\phi: \rr^+ \rightarrow \rr^+$ so that $$d_Y(f(x_1), f(x_2)) \ge \phi(d_X(x_1, x_2))$$ for all $x_1, x_2 \in X$, where $\rr^+ = [0,\infty)$. If $f$ is also Lipschitz, this is sometimes referred to in the literature as a \emph{coarse embedding}, though we will call these \emph{uniform embeddings}. If $\phi$ is a linear function, then $f$ is a \emph{quasi-isometric embedding}.  
If $G$ is a finitely generated group and $H$ is a finitely generated subgroup, the inclusion of $H$ into $G$ is a uniform embedding with respect to the word metrics on $G$ and $H$ (and this map will not be a quasi-isometric embedding if $H$ is distorted in $G$). 

 Recall that a space $X$ is \emph{uniformly contractible} if there is a function $\phi: \rr^+ \rightarrow \rr^+$ so that for every $x \in X$ the ball $B(x,R)$ contracts inside $B(x,\phi(R))$. Given a subspace $Y \subset X$, we say $X$ is \emph{uniformly contractible around $Y$} if the above holds for all points $y \in Y$. Of course, if $X$ is contractible and admits a proper, cocompact group action, then it is uniformly contractible. 

\begin{definition}
Given a finitely generated group $G$, the \emph{uniformly proper dimension} of $G$ is the minimal $n$ so that there is a contractible $n$-manifold $M^n$, equipped with a proper metric, and a uniform embedding $\rho: G \rightarrow M^n$ so that $M^n$ is uniformly contractible around $\rho(G)$. 
\end{definition}

 The uniformly contractibility assumption is essential, as Bestvina, Kapovich, and Kleiner noted that any finitely generated group has a uniform embedding into $\rr$ for some proper metric on $\rr$.
Note also that $\updim(G)$ is a quasi-isometry invariant of $G$, whereas $\actdim(G)$ is not. We record the following well-known lemma. 

\begin{lemma}
Suppose that a finitely generated group $G$ acts properly discontinuously by isometries on a proper metric space $X$. Then each orbit map $\rho: G \rightarrow X$, which takes $g$ to $gx_0$ for a choice of $x_0 \in X$, is a uniform embedding. If $X$ is contractible, then $X$ is uniformly contractible around $\rho(G)$. In particular, $\updim(G) \le \actdim(G)$. 
\end{lemma}

\begin{proof}

Since the action is properly discontinuous, the orbit map is proper. Let $g_1, \dots g_n$ be a generating set for $G$, and let $N$ be the maximum of the values $d_X(x_0, g_ix_0)$ for $i \in \{1, \dots, n\}$. Then if $g$ is a group element with $d(1,g) = m$, we have that $d_X(x_0, gx_0) \le mN$. 
Therefore, since $G$ acts by isometries, we have that $d_X(g_1x_0, g_2x_0) = d_X(g_2^{-1}g_1x_0, x_0) \le d_G(g_1, g_2)N$, so $\rho$ is $N$-Lipschitz. To prove uniform properness, let $$\phi(n) = \min_{{\substack{g \in G \\ d_G(1,g) =n}}} d(x_0, gx_0).$$ Then $\phi$ is proper since the action is properly discontinuous, and obviously $$d_X(f(1), f(g)) \ge \phi(d_G(1, g)).$$ Since $G$ acts by isometries, $d_X(f(g_1), f(g_2)) \ge \phi(d_G(g_1, g_2))$ for all $g_1, g_2 \in G$. The uniformly contractible statement is immediate as  $G$ acts cocompactly on $\rho(G)$. 
\end{proof}

Given a finitely generated group $G$, we will denote by $EG$ any contractible complex $X$ that admits a proper and cocompact cellular action by $G$. Such an $EG$ may not exist, but for the rest of the paper we will only work with groups that act on such spaces. If $G$ is torsion-free, then this is the same as the universal cover of a finite classifying space $BG$. If $G$ contains torsion, it is usually assumed that the fixed point sets of these torsion elements in $EG$ are contractible; however we do not need this assumption. We will always assume that $EG$ is equipped with a proper $G$-invariant metric.
We will need the following lemma. 

\begin{lemma}
Let $G$ be a finitely generated group, and suppose that $G$ acts properly on a contractible metric space $Y$ by isometries. Then the orbit map $\rho: G \rightarrow Y$ extends to a uniformly proper Lipschitz map $\rho: EG \rightarrow Y$ so that $d_Y(g\rho(x), \rho(gx)) < C$ for all $x \in EG$ and some constant $C >0$ (such a $\rho$ is called quasiequivariant \cite{drutukapovich}). 
\end{lemma}

\begin{proof}
 Choose a basepoint $x_0 \in EG$, and identify the orbit of $x_0$ in $EG$ with the image of the orbit map in $Y$ $G$-equivariantly. Now, since $Y$ is uniformly contractible around the image of $\rho$, we can extend the orbit map to the simplices of $EG$ so that the diameter of the image of each simplex is uniformly bounded. Since the orbit map was a uniformly embedding from $G$, this implies the extension is a uniform embedding. Since the map was $G$-equivariant on the vertices, this implies quasiequivariance for the extension. 
\end{proof}

\subsection{$\zz/2$-valued van Kampen obstruction}\label{ss:vk}
Let $\widetilde \cac(K)$ denote the simplicial configuration space of ordered pairs of distinct simplices in $K$, i.e., if $\gD$ denotes the simplicial diagonal $\gD = \{(\sigma, \tau) | \sigma \cap \tau \ne \emptyset\}$ then $$\widetilde \cac(K) = (K\times K) -\gD.$$ There is an involution $\iota$ on $\widetilde \cac(K)$ which switches the factors, let $\cac(K)$ denote the quotient. 
The double cover $\iota: \widetilde \cac(K) \to \cac (K)$ is classified by a map $c:\cac(K)\to \rr P^\infty$.
The \emph{$\zz/2$-valued van Kampen obstruction} in degree $m$ is the cohomology class $\vktwo^m(K)\in H^m(\cac(K);\zz/2)$ defined by
\[
	\vktwo^m(K) = c^*(w_1^m),
\]
where $w_1$ is the generator of $H^1(\rr P^\infty;\zz/2)$. 
If $K$ embeds into $\rr^m$, then a classifying map factors through $\rr P^{m-1}$, and hence $\vktwo^m(K) = 0$. Therefore, $\vktwo^m(K)$ is an obstruction to embedding $K$ in $\rr^m$.

Note that $\vktwo^m(K) \ne 0$ if and only if there is a cycle $\Phi \in H_m(\cac(K); \zz/2)$ so that the evaluation $\langle \vktwo^m(K), \Phi \rangle \ne 0$. 
Bestvina, Kapovich and Kleiner define a $m$-obstructor as a slight strengthening of this. 

\begin{definition}
A finite complex $K$ is an \emph{$m$-obstructor} if there is a cycle $\Phi \in H_m(\cac(K); \zz/2)$ satisfying \begin{itemize}
\item $\langle \vktwo^m(K), \Phi \rangle \ne 0$
\item If $v$ is a vertex, then the collection $\{\sigma, v\} \in \Phi$ has even cardinality.

\end{itemize}
\end{definition}

\begin{Example}
The following are examples of obstructor complexes \cite{bkk}.

\begin{itemize}
\item The disjoint union of an $m$-sphere (say triangulated as $\ast_m S^0$) and a point is an $m$-obstructor.
\item The cone on an $m$-obstructor is an $(m+1)$-obstructor.
\item The join of an $m_1$-obstructor and an $m_2$-obstructor is an $(m_1 + m_2+2)$-obstructor. 
\end{itemize}
\end{Example}

\begin{definition}
Let $\Cone_\infty(K) = K \times \rr^+/ K \times 0$. 
A proper map of $\Cone_\infty(K)$ into a metric space is \emph{expanding} if for any pair of disjoint simplices $\sigma, \tau$ in $K$, the  distance between $\sigma \times [t,\infty)$ and $\tau \times [t,\infty)$ goes to infinity as $t \rightarrow \infty$. 
\end{definition}

\begin{definition}
The \emph{obstructor dimension} of $G$, denoted $\obdim(G)$, is the maximal $n+2$ so that there is a proper expanding map $f: \Cone_\infty(K) \rightarrow EG$ where $K$ is an $n$-obstructor. 
\end{definition} 

\begin{remark}
Bestvina, Kapovich and Kleiner give a more general definition of obstructor dimension which does not involve the space $EG$ (and in particular works for all finitely generated groups). For the groups we are interested in, the two notions coincide. 
\end{remark}

The main theorem of \cite{bkk} is that $\obdim(G) \le \updim(G)$. 
It also follows from a Join Lemma for the van Kampen obstruction that $$\obdim(G_1 \times G_2) = \obdim(G_1) + \obdim(G_2).$$

\subsection{Integral van Kampen and Wu Invariants}

The trivial and nontrivial $\zz/2$-module structures on $\zz$ will be denoted by $\zz$ and $\zz^-$, respectively. Recall that if $\widetilde X$ is a complex with free $\zz/2$-action and $X$ is the quotient, then for any $\zz[\zz/2]$-module $M$, the groups $H_\ast(X, M)$ are the homology groups of the chain complex $C_\ast(\widetilde X) \otimes_{\zz/2} M$. Similarly, the groups $H^\ast(X, M)$ are the cohomology groups of the complex $\Hom_{\zz/2} (C_\ast(\widetilde X), M)$. If $M = \zz$, then these complexes can be identified with $C_\ast(X, \zz)$ and $C^\ast(X, \zz)$ respectively. If $M = \zz^-$, then $C_\ast(X, \zz^-)$ can be identified with the quotient complex $$C_\ast(\widetilde X, \zz)/(c \sim -\iota_\ast c)$$ and  $C^\ast(X, \zz^-)$ with the subcomplex $$\{f \in C^\ast(\widetilde X,\zz)| f(c) \sim -f(\iota_\ast c)\}$$

In our setting, note that the $\zz/2$-action $\iota_\ast$ on $C_\ast(\widetilde \cac(K), \zz)$ sends the chain $(\gs, \gt)$ to the chain $(-1)^{\dim \gs \dim \gt} (\gt, \gs)$.

We have 
$$
	H^i(\rr P^\infty;\zz)=
	\begin{cases}
		\zz, &\text{if $i=0$;}\\
		\zz/2, &\text{if $i>0$ and is even;}\\
		0, &\text{if $i$ is odd.}
	\end{cases}
$$
The sequence $0 \rightarrow \zz \rightarrow \zz[\zz/2] \rightarrow \zz^- \rightarrow 0$ induces a long exact sequence in cohomology. Since $H^*(\rr P^\infty;\zz[\zz/2]) = H^*(S^\infty, \zz) = 0$ for $\ast > 0$, we have
$$
	H^i(\rr P^\infty;\zz^-)=
	\begin{cases}
		\zz/2, &\text{if $i$ is odd;}\\
		0, &\text{if $i$ is even.}
	\end{cases}
$$
Let $e_1$ denote the nontrivial element of $H^1(\rr P^\infty; \zz^-) \cong \zz/2$. For the rest of the paper we will let $\geps$ denote the sign of $(-1)^n$, so that $e_1^n$ is the nontrivial element of $H^n(\rr P^\infty; \zz^\geps)\cong \zz/2$.
For $K$ a finite complex, the \emph{integral degree $n$ van Kampen obstruction}, denoted $\vk^n(K)$, is  given by
$$
	\vk^n(K) = c^*(e_1^n) \in H^n(\cac(K);\zz^\geps)
$$

Now, if $f$ is an embedding of $K$ into $\rr^{n+1}$, then $f$ determines a $\zz/2$-equivariant \emph{Gauss map} $\widetilde F$ from $\widetilde \cac(K)$ to $\mathbb{S}^n$; $$\widetilde F((x,y)) = \frac{f(x) - f(y)}{||f(x)-f(y)||}.$$ There is an induced map $F:\cac(K) \rightarrow \rr P^n$. Again by the coefficient long exact sequence, $$H^n(\rr P^n; \zz^{-\geps}) \cong \zz.$$ Let $\eta$ denote the generator of $H^n(\rr P^n; \zz^{-\geps})$. The \emph{Wu invariant} of $f$, denoted $\Wu_f^n(K)$, is the pullback \[\Wu_f^n(K) = F^\ast(\eta) \in H^n(\cac(K);\zz^{-\geps}).\] 

If $f$ and $g$ are two embeddings with $\Wu_f^n(K) \ne \Wu_g^n(K)$, then $f$ and $g$ are not ambient isotopic (since such an isotopy would induce a $\zz/2$-equivariant homotopy between the Gauss maps $\widetilde F$ and $\widetilde G$). 
More interestingly, $\Wu$ is a complete ambient isotopy invariant for embedding of $n$-complexes into $\rr^{2n+1}$ for $n > 1$. It is easy to construct examples of embeddings with $\Wu_f \ne \Wu_g$. For  example, if $K$ is the disjoint union of two circles, then the degree two Wu invariant evaluated on the fundamental class of $\cac(K)$ measures the linking number of an embedding of $K$ into $\rr^3$.

 On the other hand, the class $\vk$ does not depend on the embedding. See Figure 1 for an explanation of how the $\Wu$ invariant evaluates differently on cells than the van Kampen obstruction. Roughly, both invariants admit geometric representatives obtained by taking an embedding into $\rr^{n+1}$, generically projecting to $\rr^n$, and then counting signed intersections between disjoint simplices with $\dim \gs + \dim \gt = n$. The $\Wu$ invariant is more refined as it remembers which simplices are ``higher" from the point of view of the projection. 

There are obvious homomorphisms $\zz^{+/-} \rightarrow \zz/2$. Under these change of coefficients, $e_1 \in H^1(\rr P^\infty, \zz^-)$ maps to $\omega_1$, and $\eta \in H^n(\rr P^n, \zz^{-\geps})$ maps to $\omega_1^n$. Therefore, both the integral van Kampen obstruction and the Wu invariant  reduce to $\vk^n_{\zz/2}$ via change of coefficients. 

There is also a natural evaluation map $$\langle \hspace{.5mm}, \rangle: H^n(\cac(K), \zz^{-\geps}) \times H_n(\cac(K), \zz^{-\geps}) \rightarrow \zz^{-\geps}$$ which comes from the identifications $$C^n(\cac(K), \zz^{-\geps}) = \Hom_{\zz/2}(C_n(\cac(K), \zz), \zz^{-\geps}) \cong \Hom_{\zz/2}(C_n(\cac(K), \zz) \otimes_{\zz/2} \zz^{-\geps}, \zz^{-\geps}).$$
Therefore, a representative for a cohomology class gives a $\zz/2$-homomorphism from $C_n(K, \zz^{-\geps})$ to $\zz^{-\geps}$. 
This passes to a well-defined homomorphism on cohomology and homology.  
We record the following lemma.

\begin{lemma}
Let $\phi \in H^n(\cac(K), \zz^{-\geps})$ and $\psi \in H_n(\cac(K), \zz^{-\geps})$, let $p^{-\geps}$ be the nontrivial homomorphism $\zz^{-\geps} \rightarrow \zz_2$, and let $\phi_{\zz/2}$ and $\psi_{\zz/2}$ denote the images in $H_n(\cac(K), \zz/2)$ and $H_n(\cac(K), \zz/2)$ of $\phi$ and $\psi$ respectively. Then $$p^{-\geps}(\langle \phi, \psi \rangle) = \langle \phi_{\zz/2}, \psi_{\zz/2} \rangle.$$
\end{lemma}



\begin{figure}
\centering
\begin{tikzpicture}
\node[above] at (-3, 0) {$\gs$};
 \draw[thick, ->] (-2,.1) -- (-2,1);
  \draw[thick] (-2,-1) -- (-2,-.1);
  \node[right] at (-2, -1) {$\gt$};
 \draw[thick, ->] (-3,0) -- (-1,0);
 
 \node[above] at (0, 0) {$\gs'$};
  \draw[thick, ->] (1,-1) -- (1,1);
 \draw[thick] (0,0) -- (.9,0);
  \draw[thick, ->] (1.1,0) -- (2,0);
    \node[right] at (1, -1) {$\gt'$};
    
    \node at (-.2, -2) {$\vk(\gs, \gt) = \vk(\gs', \gt') = -\vk(\gt, \gs)$};
    \node at (-.2, -2.5) {$\Wu(\gs, \gt) = -\Wu(\gs', \gt') = \Wu(\gt, \gs)$};

\end{tikzpicture}
\caption{An example of how the van Kampen obstruction and Wu invariant evaluate on a pair of cells. In this case, these edges are part of an embedded graph in $\rr^3$. In both cases, $\vk$ and $\Wu$ switch sign upon changing the orientation of $\gs$ or $\gt$. In this case,  $\iota(\gs, \gt) = -(\gt,\gs)$, so $\Wu(\gs, \gt) = \Wu(\gt, \gs)$ and $\vk(\gs, \gt) = -\vk(\gt, \gs)$. 
}

\end{figure}

\section{Equivariant obstructors}\label{s:equivariant}

We now use the Wu invariant to obstruct certain equivariant embeddings of finite complexes into Euclidean space.  In the next section, we will develop a coarsened version which obstructs properly discontinuous actions. 

Suppose that $H$ is a finite group and $K$ is an $H$-complex. Suppose that $\rho: H \rightarrow \Homeo(\rr^{n+1})$ is a homomorphism.
An $H$-equivariant embedding of $K$ is an embedding $f: K \rightarrow \rr^{n+1}$ satisfying $$f(hk) = \rho(h)(f(k)) \text{ for all } k \in K, h \in H.$$ 
If $\rho(h)$ is orientation-preserving, it is isotopic to the identity \cite{kirby}, so if $f$ is $H$-equivariant than $\Wu_f = \Wu_{\rho(h) \circ f} = \Wu_{f \circ h}$ for all $h \in H$. We will want to assume $\rho(h)$ preserves orientation, so we consider only the elements of $H$ that are squares. For $H$ a finite group, let $S(H)$ be the subgroup generated by the squares of elements of $H$. 
For a cycle $$\Phi = \sum_{(\gs,\gt) \in \widetilde \cac(K)} (\gs, \gt) \otimes a_{(\sigma, \tau)} \in H_n(\cac(K); \zz^{-\geps}),$$ let $\Phi_{\zz/2}$ denote its image in $H_n(\cac(K); \zz/2)$ under the change of coefficients homomorphism. 

\begin{definition}
Let $K$ be a finite complex. A cycle $\Phi$ in $H_n(\cac(K); \zz^{-\geps})$ is an \emph{evaluation $n$-cycle} if 
\begin{itemize}
\item $\langle \vktwo^n(K), \Phi_{\zz/2}\rangle \ne 0$
\item If $\sigma$ is an $n$-cell, then the sum $$\sum_{\substack{v \in K^{(0)}\\ (\sigma, v) \in \widetilde C(K)}} a_{(\gs,v)} = 0.$$
\end{itemize}
\end{definition}

As in the definition of obstructor, the first condition is the more important one, and the second condition guarantees that certain join formulas will hold. 
Since $\Wu_f^n$ maps to $\vktwo^n$ under change of coefficients, we have the following lemma. 

\begin{lemma}\label{l:nontrivial}
If $f:K \rightarrow \rr^{n+1}$ is an embedding and $\Phi \in H_n(\cac(K); \zz^{-\geps})$ is an evaluation $n$-cycle, then $\Wu_f^n(K)$ evaluates nontrivially on $\Phi$. 
\end{lemma}

It will be convenient later to have the following refinement of $\obdim(G)$.
\begin{definition}
The \emph{$\zz$-valued obstructor dimension} of $G$, denoted $\obdim_{\zz}(G)$, is the maximal $n+2$ so that there is a proper expanding map $f:\Cone_\infty(K) \rightarrow EG$ where $K$ is an evaluation $n$-cycle. 
\end{definition}

\begin{remark}
It is obvious that $\obdim_\zz(G) \le \obdim(G)$. On the other hand, many of the complexes used to compute obstructor dimension contain evaluation cycles that reduce to the nontrivial $\zz/2$-valued obstructor cycles. For example, if $K$ is the $n$-fold join of $3$ points, then $H_n(\cac(K), \zz^{-\geps})$ surjects onto $H_n(\cac(K), \zz/2)$ (this follows for example from the Join Lemma below). Note that in this case $H_n(\cac(K), \zz^{\geps}) = 0$. 
\end{remark}

If $H$ acts on $K$ cellularly, then there is an induced action on $\cac(K)$, and hence an action on $H_n(\cac(K); \zz^{-\geps})$. We always assume that $H$ acts trivially on $\zz^{-\geps}$. 

\begin{definition}
An $H$-complex $K$ is an \emph{equivariant $(n+1)$-obstructor} if
there exists an evaluation cycle $\Phi \in H_n(\cac(K); \zz^{-\geps})$ and a subset $A \subset S(H)$ so that  $$\sum_{h \in A}h_\ast\Phi = 0.$$
 \end{definition}

\begin{lemma}
Suppose that $K$ is a finite $H$-complex and $f: K \rightarrow \rr^n$ is an embedding. Then $\Wu_{f \circ h} = h^\ast(\Wu_f)$. In particular, for all $\Phi \in H_n(\cac(K); \zz^{-\geps})$, $$\langle \Wu_f^n(K), h_\ast\Phi \rangle = \langle \Wu_{f \circ h}^n(K), \Phi \rangle.$$
\end{lemma}

\begin{proof}
This follows immediately from the equalities $$h^\ast \Wu_f^n((\gs, \gt)) = \Wu_f^n((h\gs, h\gt)) = \Wu^n_{f \circ h}(\gs, \gt).$$
\end{proof}

\begin{lemma}
If $K$ is an equivariant $(n+1)$-obstructor, then $K$ does not equivariantly embed into $\rr^{n+1}$. 
\end{lemma}
\begin{proof}

Suppose there was a representation $\rho: H \rightarrow \Homeo(\rr^n)$ and an equivariant embedding $f: K \rightarrow \rr^n$. Every $h \in S(H)$ acts by orientation-preserving homeomorphisms on $\rr^n$, so in particular $f$ is isotopic to $f \circ h = \rho(h) \circ f$. Let $A \subset S(H)$ with $\sum_{h \in A}h_\ast\Phi = 0$, where $\Phi$ is an evaluation cycle. We then have that $$0 = \langle \Wu_f^n(K), (\sum_{h \in A} h_\ast\Phi) \rangle =  \sum_{h \in A} \langle\Wu_f^n(K), h_\ast\Phi \rangle$$ $$= \sum_{h \in A}\langle \Wu^n_{f \circ h}(K), \Phi \rangle = |A|\langle \Wu_f^n(K), \Phi \rangle \ne 0$$ where the last inequality is by  Lemma \ref{l:nontrivial}. This is a contradiction. 
\end{proof} 

Note that if $K$ is an equivariant $(n+1)$-obstructor then $K$ does not embed into $\rr^n$.

\subsection{Examples of equivariant obstructors}

We now give our main examples of finite complexes which are equivariant $(n+1)$-obstructors. All our examples are iterated joins or cones of the following. 

\begin{lemma}\label{l:mainex}
Suppose that $\zz/4 = \langle h \rangle$ acts on $K = 5$ points by cyclically permuting $4$ of the points and fixing one. Then $K$ is an equivariant $1$-obstructor.
\end{lemma}

\begin{proof}

Let $A$ denote the fixed point and $\{1,2,3,4\}$ the other points. Let $\Phi \in H_0(\cac(K); \zz^-)$ be given by $$\Phi = \{(A,1) \otimes 1, (1,3) \otimes 1, (3,A) \otimes 1\}.$$ Then $\langle \vkt(K), \Phi_{\zz/2} \rangle \ne 0, \Phi + h^2_\ast \Phi = 0 \in H_0(\cac(K); \zz^-)$, and $\Phi$ satisfies the second condition of equivariant obstructors.
\end{proof}

In the next two lemmas, we are explicitly identifying $H_n(\cac(K), \zz^{-\geps})$ with the homology of the quotient complex $C_n(\widetilde \cac(K), \zz)/ (c \sim (-1)^{-\geps} \iota_\ast c)$. The next two lemmas mirror the Cone Lemma and the Join Lemma from \cite{bkk}, but unfortunately we have to keep track of signs.
We will let $\Cone K$ denote the finite cone $K \times [0,1]/ K \times 0$, in order to distinguish it from $\Cone_\infty(K)$.  

\begin{ConeLemma}
Suppose that $K$ is an $H$-complex. Let $\Cone K$ be the cone of $K$, and extend the action of $H$ by fixing the cone point.  If $K$ is an equivariant $n$-obstructor, then $\Cone K$ is an equivariant $(n+1)$-obstructor.
\end{ConeLemma}

\begin{proof}
Let $\Phi \in H_n(\cac(\Cone K); \zz^{-\geps})$ be an evaluation cycle, and
define $\Phi' \in C_{n+1}(\cac(\Cone K); \zz^{-\geps})$ by $$\Phi' = \{(-1)^{\dim \sigma}(\sigma, \Cone(\tau)) - (\Cone(\sigma), \tau)) | (\sigma, \tau) \in \Supp \Phi\}$$ and then extending linearly. 

We first claim that this procedure is well-defined. If $(\sigma, \tau) \in \Phi$, then $$(\sigma, \tau) \sim (-1)^{\dim \gs \dim \gt + n-1}(\tau, \gs).$$ which produces $$ (-1)^{\dim \gs \dim \gt + \dim \gs -1}(\tau, \Cone(\gs)) +  (-1)^{\dim \gs \dim \gt + n}(\Cone(\tau), \gs))$$

So, we are done by the equalities $$\dim \gs + \dim \gs (\dim \gt + 1) + n = \dim \gs \dim \gt + n\mod 2$$  $$(\dim(\sigma) + 1)\dim \tau + 1 = \dim \gs \dim \gt + \dim \gs -1 \mod 2 $$

We now check that $\Phi'$ is a cycle. Since $\partial(\sigma, \tau) = (\partial \gs, \tau) + (-1)^{\dim \gs}(\gs, \partial \gt)$, we have exactly set it up so that $\partial\Phi' (\gs, \gt) = 0$. Since $\Phi$ is an evaluation cycle, for a cell $(\sigma, \Cone(\alpha) )$, we have $$\sum_{\alpha \subset \tau} a_{\sigma, \tau} = 0.$$ The cells $(\sigma, \Cone(\tau))$ are the cells containing $(\sigma, \Cone(\alpha) )$, and since these are multiplied by the same constant, we have  $d\Phi'((\sigma, \Cone(\alpha))= 0$. For the nontriviality condition, $\Phi'_{\zz/2}$ is precisely the chain constructed in the Cone Lemma  of \cite{bkk}, which was shown to have $\vktwo^n(\Phi'_{\zz/2}) \ne 0$ as long as $\vktwo^n(\Phi_{\zz/2}) \ne 0$. If $\sum_{h \in A} h\Phi = 0$ then obviously $\sum_{h \in A} h\Phi' = 0$. Finally, it is straightforward to check that $\Phi$ satisfies the second condition of equivariant obstructors if $\Phi$ does. 
\end{proof}

\begin{JoinLemma}
Let $K$ be an $H$-complex which is an equivariant $(n + 1)$-equivariant obstructor, and suppose that $J$ is a complex with an evaluation $m$-cycle. Let $H$ act on $K \ast J$ by permuting the $K$-factor and fixing $J$. Then $K \ast J$ is an equivariant $(n + m + 3)$-obstructor.
\end{JoinLemma}

\begin{proof}

Let $\Phi_K$ be an evaluation cycle in $H_{n}(\cac(K); \zz^{(-1)^{n}})$ with $$\sum_{h \in A} h_\ast(\Phi_K) = 0$$ and let $\Phi_J$ be the evaluation $m$-cycle in $H_{n}(\cac(J); \zz^{(-1)^{m}})$. 
Form a new cycle $\Phi \in H_{n+m+2}(\cac(K \ast J); \zz^{(-1)^{n+m+2}})$ by putting in for every $(\gs, \gt) \in \Phi_K$ and $(\gs', \gt') \in \Phi_J$ the chains $$(-1)^{\dim \gs(\dim\gt' + 1) + \dim\gs'}(\gs \ast \gs', \gt \ast \gt')\hspace{1mm} + \hspace{1mm} (-1)^{\dim (\gs + 1)(\dim \gs' + 1) + \dim \gs' \dim \gt'}(\sigma \ast \tau', \gt \ast \gs')$$ and extending linearly. 

We need to show that this procedure is well-defined after passing to the quotient complexes $C_\ast(\widetilde \cac(K_i))/\iota$. In particular, we need to show that the involuted $(\gs, \gt)$ and $(\gs', \gt')$ produce equivalent chains. For these calculations, we will shorten $\dim \gs$ in the exponents to $\gs$. 

\begin{itemize}[leftmargin=*]

\item $(\gs, \gt) \text{ and } (\gt', \gs')$. In this case, the procedure gives $$(-1)^{\gs(\gs'+1) + \gt'} (\sigma \ast \tau', \tau \ast \sigma') + (-1)^{(\gs+1)(\gt' + 1) + \gt'\gs'}(\gs \ast \gs', \gt \ast \gt').$$

Since $(\gt', \gs') \cong (-1)^{\gs' \gt' + \gt' + \gs' + 1} (\gs', \gt')$  we need to check that: $$\gs' \gt' + \gt' + \gs' + 1 + (\gs+1)(\gt' + 1) + \gt'\gs' = \gs (\gt'+1) + \gs' \mod 2$$ $$(\gs + 1)(\gs' + 1) + \gs'\tau' = \gs (\gs' + 1) + \gt'  + \gs'\gt' + \gs' + \gt' + 1\mod 2$$
which is easily verifed.

 \item $(\gt, \gs) \text{ and } (\gs', \gt')$  In this case, the procedure gives $$(-1)^{\gt(\gt'+1) + \gs'}(\gt\ast\gs', \gs \ast \gt') + (-1)^{(\gt+1)(\gs'+1) + \gs'\gt'}(\gt\ast\gt', \gs\ast\gs').$$ Since $(\gs \ast \gs', \gt \ast \gt) \cong (-1)^{(\gs + \gs')(\gt + \gt')}$, we need to show that 
$$\gs(\gt'+1) + \gs' + (\gs + \gs')(\gt + \gt') = \gs\gt + \gs + \gt + 1 + (\gt+1)(\gs' + 1) + \gs'\gt' \mod 2$$ 
$$(\gs + 1)(\gs' + 1) + \gs'\gt' + (\gs + \gt')(\gt + \gs') =   \gs\gt + \gs + \gt + 1 +\gt(\gt'+1) +\gs'\mod 2$$
which again is easily verified.

\end{itemize}

Again, $\Phi_{\zz/2}$ is the cycle constructed in the Join Lemma of \cite{bkk}.
To see that $\Phi$ is  a cycle, assume that we have a $(n_1 + n_2 +1)$-cell in $K_1 \ast K_2$. We can assume without loss of generality that this cell is of the form $(\sigma_1 \ast \alpha_2, \tau_1 \ast \tau_2)$, where $\dim(\gs_1) + \dim(\gt_1) = n_1$ and $\dim(\ga_2) + \dim(\gt_2) = n_2 -1$. 
This cell is contained precisely in the cells $(\sigma_1 \ast \sigma_2, \tau_1 \ast \tau_2)$ where $\alpha_2 \subset \sigma_2$. Since $\Phi_2$ is a cycle, we have that the sum $$\sum_{\substack{\sigma_2, \tau_2 \in \Phi_2 \\ \alpha_2 \subset \sigma_2}}a_{(\gs_2, \gt_2)} = 0.$$ Since for each of these cells, $ (\sigma_1 \ast \sigma_2, \tau_1 \ast \tau_2)\otimes a_{\sigma_1 \tau_1}a_{\sigma_2\tau_2} \in \Phi$ and have the same sign, it follows that $\Phi$ is a cycle (if $\alpha_2 = \emptyset$ then we require the second item in the definition of evaluation cycle to prove this).

Now, suppose that $\sum_{h \in A} h_\ast\Phi_K = 0$. Then for any $(\gs, \gt) \in \cac(K)$, the sum $\sum_{h \in A} a_{h\gs,h\gt} = 0$. This immediately implies that for any $(\gs \ast \gs', \gt \ast \gt')$, the sum $\sum_{h \in A} a_{h(\gs \ast \gs'), h(\gt \ast \gt')} = 0$, so $\Phi = 0$. 
The second condition of an equivariant obstructor is trivially satisfied since no simplices in $\Phi$ are paired with vertices.
\end{proof}

\begin{remark}
The usual homological tool to analyze group actions on $\mathbb{S}^n$ or $\rr^n$ is \emph{Smith theory}, and this handles far more examples than our method does. For example, if $p$ is a prime, then the fixed set of a orientation preserving $\zz/p$-action on $S^n$ is a homology $r$-sphere with $r < n-1$. This immediately obstructs $(p+1)$-points with $\zz_p$-action as above equivariantly embedding into $S^1$, as well as all joins of this complex with the product action embedding into $S^{2n+1}$. 

 The reason that we do not use Smith theory is that we do not have an adequate version of coarse Smith theory that could handle the examples of groups that we were interested. A coarse version of Smith theory has been developed by Hambleton and Savin \cite{hs},
but it does not seem to be applicable to our examples. In particular, they relate the coarse topology of an ambient $G$-space $X$ to the coarse topology of a ``bounded fixed point set''. This consists of points in $X$ which are fixed up to bounded distance by every element of $G$, which in our main examples (when $G$ is torsion-free) is always empty. In Section 6 we will consider some examples with torsion, and the methods of \cite{hs} do probably obstruct actions on uniformly contractible manifolds.  
\end{remark}

\begin{remark}
This use of the $\Wu$ invariant is our attempt to build an ``equivariant van Kampen obstruction". A natural place for such an invariant to live is in the equivariant cohomology group $H^\ast_H(\cac(K), \zz^{\pm \geps}$), but we couldn't make this work. One difficulty is that if the $H$-action on $\rr^{n+1}$ is not affine, then there is not an induced $H$-action on $\rr P^n$. The larger problem is that the usual applications of equivariant obstruction theory require knowing both the $H$-action on the domain and range, whereas we are only given the $H$-action on the domain. 
\end{remark}

\section{Coarse Wu Invariant}\label{s:coarsewu}

Let $K$ be a finite complex. 
Equip $\rr^{n+1}$ with a proper metric, and suppose that $f: \Cone_\infty K \rightarrow \rr^{n+1}$ is a proper, expanding map. Then there are induced maps $f_t: \Cone K \rightarrow \rr^{n+1}$ defined by $$f_t(x,s) = f(x, st) \text{ for } s \in [0,1], t \in [0,\infty)$$ 

Again, $\Cone K$ here denotes the finite cone $K \times [0,1]/K \times 0$. The basic idea behind defining the coarse Wu invariant is that if $f$ is a proper, expanding map, the $f_t$ will eventually be \emph{almost embeddings}, where an almost embedding maps disjoint simplices of the cone disjointly. An almost embedding suffices to define the $\Wu$ invariant, and for large enough $t$ this will stabilize to give a well-defined class in $H^n(\cac(\Cone K), \zz^{-\geps})$. We also want the $\Wu$ invariant to not change if we postcompose $f$ with a homeomorphism of $\rr^n$ which is isotopic to the identity. This composition may no longer be expanding, but will be an almost embedding, and furthermore will be isotopic to $f$ via almost embeddings. Therefore, we will eventually define the coarse Wu invariant for all maps $f: \Cone_\infty(K) \rightarrow \rr^n$ which are isotopic to proper, expanding maps. 

\begin{lemma}
Let $f: \Cone_\infty K \rightarrow \rr^{n+1}$ be a proper, expanding map. Then there exists a $T_f > 0$ so that for all $t > T_f$, $f_t$ is an almost embedding.
\end{lemma}
\begin{proof}
By the definition of expanding, there exists $T_f' > 0$ so that  $\sigma \times [T_f',\infty)$ and $\tau \times [T_f',\infty)$ are disjoint for each pair of disjoint simplices $(\sigma, \tau)$. 
Since $f$ is proper, there exists $T_f \ge T_f'$ so that $\sigma \times [0, T_f'] \cap \tau \times [T_f, \infty]$ for each pair of $\sigma$ and $\tau$ as above. Therefore, $f_t$ for all $t > T_f$ is an almost embedding. 
\end{proof}

Therefore, for $f$ proper and expanding, there exists $T_f > 0$ so that for all $t > T_f$, there is a well-defined Gauss map $F_t: \widetilde \cac(\Cone K) \rightarrow \mathbb{S}^n$ $$\widetilde F_t((x,s), (y,s')) = \frac{f_t(x,s) - f_t(y, s')}{||f_t(x,s) - f_t(y, s')||}$$ where either $s$ or $s' = 1$ since $\widetilde \cac(\Cone K)$ is the simplicial configuration space.

\begin{lemma}\label{l:coarsewu1}
For all $t, t' > T_f$, $\widetilde F_t$ and $\widetilde F_{t'}$ are $\zz/2$-equivariantly homotopic. 
\end{lemma}

\begin{proof}
There is an obvious homotopy of $f_t$ to $f_{t'}$ by the $\{f_s\}_{s \in [t,t']}$. Since each $f_s$ induces a well-defined Gauss map $F_s$, these give a $\zz/2$-equivariant homotopy between $\widetilde F_t$ and $\widetilde F_{t'}$. 
\end{proof}

We will say that $f,f': \Cone_\infty(K) \rightarrow \rr^n$ are \emph{isotopic} if there is an ambient isotopy $\{j^s\}_{s \in [0,1]}$ of $\rr^n$ with $j_0$ the identity and $j_1 \circ f = f'$. In particular, we are not assuming $f$ and $f'$ are embeddings. 

\begin{lemma}\label{l:coarsewu2}
Suppose that $f: \Cone_\infty(K) \rightarrow \rr^{n+1}$ and  $f': \Cone_\infty(K) \rightarrow \rr^{n+1}$ are isotopic maps and that $f$ is proper and expanding. Then for $T_{f}$ as above and $t > T_{f}$, we have that the Gauss maps $\widetilde F_t$ and $\widetilde F'_t$ are $\zz/2$-equivariantly homotopic. 
\end{lemma}

\begin{proof}
Let $\{j^s\}_{s \in [0,1]}$ be an isotopy between $f$ and $\bar f$. The homeomorphisms $j^s$ preserve disjointness of simplices, so for each $t > T_{f}$, the map $j^s_t := j^s \circ f_t$ induces a well-defined Gauss map $J^s_t$. The $\{J^s_t\}_{s \in [0,1]}$ give a $\zz/2$-equivariant homotopy between $\widetilde F_t$ and $\widetilde F'_t$. 
\end{proof}

\begin{definition}
Suppose that $\bar f: \Cone_\infty K \rightarrow \rr^{n+1}$ is isotopic to a proper expanding map $f: \Cone_\infty K \rightarrow \rr^{n+1}$, and $T_{f} > 0$ is defined as above. Suppose that $\widetilde F_t: \widetilde \cac(\Cone K) \rightarrow \mathbb{S}^n$ is the induced Gauss map for $f_t$ as above, and let $F_t$ denote the induced map $\cac(\Cone K) \rightarrow \rr P^n$. Let $\eta$ be the generator of $H^n(\rr P^n; \zz^{-\geps})$. 
The \emph{coarse Wu invariant} of $\bar f$, denoted $\CWu_{\bar f}^n(K)$, is defined to be $$ F_t^*(\eta) \in H^n(\cac(\Cone K); \zz^{-\geps})$$ for $t > T_f$. 
\end{definition}

The coarse Wu invariant is well-defined (i.e. does not depend on $t$ or $f$)  by Lemmas \ref{l:coarsewu1} and \ref{l:coarsewu2}. 

\begin{lemma}\label{l:coarsewu4}
Suppose that $f: \Cone_\infty K \rightarrow \rr^{n+1}$ is a proper, expanding map, and $g$ is an orientation-preserving homeomorphism of $\rr^{n+1}$. Then $\CWu_f^n(K) = \CWu_{g \circ f}^n(K)$. 
\end{lemma}

\begin{proof}
Since $g$ is orientation-preserving, it is isotopic to the identity. Therefore, $f$ and $g \circ f$ are isotopic, so the coarse $\Wu$ invariant is defined for $g \circ f$ and equal to $\CWu^n_f(K)$. 
\end{proof}

\begin{Theorem}\label{l:coarsewu3}
 Let $f: \Cone_\infty K \rightarrow \rr^{n+1}$ be a proper, expanding map.  Suppose that $f':\Cone_\infty K \rightarrow \rr^{n+1}$ is properly homotopic to $f$ through maps $\{h^s\}_{s \in [0,1]}$ so that for each  $\gs \subset K$, \begin{equation}h^s(\Cone_\infty \sigma) \subset N_R(f(\Cone_\infty \sigma))\tag{$\dagger$}\end{equation} for some $R > 0$.  Then $\CWu^n_f(K) = \CWu^n_{f'}(K)$.
\end{Theorem}

\begin{proof}

Choose $T > 0$ so that  $\gs \times [T, \infty)$ and $\gt \times [0, \infty)$ have distance $> 2R$ for each pair of disjoint simplices $\gs$ and $\gt$ in $K$.
Since the homotopy is proper and satisfies $(\dagger)$, there exists a $T' > T$ so that $h^s(\gs \times [T',\infty)))$ is contained in the $R$-neighbourhood of $f(\gs \times [T, \infty))$. Otherwise, there would be a sequence of points $x_i$ with $x_i \rightarrow \infty$ in $\Cone_\infty(K)$ and $h_s(x_i) \in B_{R+T}(f(x,0))$, contradicting properness. 

Therefore, since  $h^s(\gt \times [0, \infty) \subset N_R(f(\gt \times [0,\infty)))$, for disjoint simplices $\sigma$ and $\tau$ in $K$, $h^s(\gs \times T')$ and  $h^s(\gt \times [0, \infty))$ are disjoint for all $s$.  So, for all $t > T'$, the $\{h^s_t\}_{s \in [0,1]}$ induce a well-defined Gauss map from $\widetilde \cac(\Cone K)$ to $\mathbb{S}^n$. These maps provide a homotopy between $F_t$ and $F'_t$. 
\end{proof}

Note that if $K$ is a complex with an evaluation cycle $\Phi \in H_{n-1}(\cac(K); \zz^{-\geps})$, and $f: \Cone_\infty(K) \rightarrow \rr^{n+1}$ is a proper expanding map, then by the Cone Lemma and Lemma \ref{l:nontrivial}, there is an evaluation cycle $\Phi' \in H_n(\cac(\Cone K; \zz^{-\geps})$ that $\CWu^n_f$ evaluates nontrivially on. 

\begin{lemma}\label{bdddistance}
Suppose that $f: \Cone_\infty(K) \rightarrow \rr^{n+1}$ is a proper expanding map, that $\rr^{n+1}$ is uniformly contractible around the image of $f$, and $f': \Cone_\infty(K) \rightarrow \rr^{n+1}$ is uniformly bounded distance from $f$, i.e. $\exists C> 0$ so that $$d(f(x,s), f'(x,s)) < C \text{ for all } (x,s) \in \Cone(K).$$ Then $\CWu_f^n(K) = \CWu_{f'}^n(K))$. 
\end{lemma}

\begin{proof}
Since $\rr^n$ is uniformly contractible around the image of $f$, we can homotope $f$ to $f'$ so that points move a uniformly bounded distance during the homotopy (say $< R$). As before, choose $T > 0$ so that $\gs \times [T, \infty)$ and $\gt \times [0, \infty)$ have distance $> 2R$ for each pair of disjoint simplices $\gs$ and $\gt$ in $K$. This guarantees that a homotopy exists between the Gauss maps for $f_t$ and $f_t'$ for large enough $t$, and hence the coarse Wu invariants are the same. 
\end{proof}

\section{Equivariant obstructor dimension}\label{s:eqobdim}
We now show that the coarse Wu invariant obstructs proper, expanding maps $\Cone_\infty(K) \rightarrow \rr^n$ that are ``coarsely equivariant", and hence obstructs properly discontinuous actions on $\rr^n$. Our notion of coarse equivariance is different from quasi-equivariance as defined in Section 2. For example, we want different groups acting on the domain and range (a finite group for the domain and usually a torsion-free group for the range), and we also want to allow proper homotopies that preserve disjointness of far away simplices. 

\begin{definition}
Let $H$ be a finite group, $K$ a finite $H$-complex, and extend the $H$-action to $\Cone_\infty(K)$ by acting trivially on $[0,\infty$. Let $G$ be a group and $EG$ a contractible, proper, cocompact, $G$-complex. A proper map $f: \Cone_\infty K \rightarrow EG$ is \emph{$H$-preserving} if for each $h \in H$, there is $g_h$ in $G$ and $R > 0$ so that 
\begin{itemize}
\item $f\circ h$ is homotopic via a proper homotopy $\{j^s\}_{s \in [0,1]}$ to $g_h \circ f$. 
\item For all $\sigma \in K$ and $s \in [0,1]$, $j^s(\Cone_\infty(\sigma))$ is contained in the $R$-neighbourhood $N_R(f \circ h(\Cone_\infty \sigma))$. 
\end{itemize}
\end{definition}

In particular, this implies that $g_h \circ f(\Cone_\infty \sigma)$ is in the $R$-neighborhood of  $f\circ h(\Cone_\infty \sigma)$. 
Of course, the element $g_h$ may not be unique. We can and will assume that for all $h \in S(H)$, the elements $g_h$ are in $S(G)$. 

\begin{definition}
The \emph{equivariant obstructor dimension} of $G$ is the maximal $n+3$ so that there is an $H$-equivariant $(n+1)$-obstructor $K$ and an $H$-preserving proper expanding map $f: \Cone_\infty K \rightarrow EG$. 
\end{definition}

From the definitions, we have the following lemma: 
\begin{lemma}
$\obdim(G) \le \eqobdim(G) \le \obdim(G) + 1$. 
\end{lemma}

The following is our main theorem. 

\begin{Theorem}
$\eqobdim(G) \le \actdim(G)$. 
\end{Theorem}

\begin{proof}
Suppose that $\eqobdim(G) = n+1$ and $G$ acts properly on a contractible $n$-manifold $M^n$. 
We first assume $M^n$ is homeomorphic to $\rr^{n}$. We equip $M^n$ with a proper $G$-invariant metric. 
By assumption, we get a quasi-equivariant map $\rho: EG \rightarrow M^n$, which is uniformly proper and Lipschitz. 
 By precomposing with the $H$-preserving map $f: \Cone_\infty K\rightarrow EG$, we get a proper expanding map $\rho \circ f: \Cone_\infty K \rightarrow M$, where $K$ is an equivariant $(n-1)$-obstructor. 
 
We now show that $\CWu_{\rho \circ f}^n(K) = \CWu^n_{\rho \circ f \circ h}(K)$ for all $h \in S(H)$. 
We have by Lemma \ref{l:coarsewu4} that for all $h \in S(H)$,  $$\CWu^n_{\rho \circ f}(K) = \CWu^n_{g_h \circ \rho \circ f}(K)$$ since $g_h \in S(G)$ and hence $g_h: M \rightarrow M$ is orientation preserving for all $h \in S(H)$, and hence isotopic to the identity. 
 
 We have that $\rho \circ f \circ h$ is properly homotopic to $\rho \circ g_h \circ f$ through the maps $\rho \circ j^s$. Since $j^s(\Cone_\infty(\gs))$ is contained in $N_R(f \circ h (\Cone_\infty(\gs))$ and $\rho$ is Lipschitz, we have that $\rho \circ j^s(\Cone_\infty(\gs))$ is contained in $N_{R'}(\rho \circ g_h \circ f(\Cone_\infty(\gs))$. Therefore, by Lemma \ref{l:coarsewu3}, $\CWu^n_{\rho \circ f \circ h} = \CWu^n_{\rho \circ g_h \circ f}$.
 
 Since $\rho$ is quasiequivariant, we have that $\rho \circ g_h \circ f$ is uniformly bounded distance from $g_h \circ \rho \circ f$. Since $M^n$ is uniformly contractible around the image of $\rho(G)$, Lemma \ref{bdddistance} implies that $\CWu_{\rho \circ g_h \circ f} = \Wu_{g_h \circ \rho \circ f}$.

So, therefore $$ \CWu^n_{\rho \circ f \circ h}(K) = \CWu^n_{g_h \circ \rho \circ f}(K) = \CWu^n_{\rho \circ f}(K)$$ for all $h \in S(H)$. 
 
 Now, by the Cone Lemma, we have an evaluation cycle $\Phi \in H_{n}(\Cone K; \zz^{-\geps})$ which $\CWu^n_{\rho \circ f}(K)$ evaluates nontrivially on. 
 
As in the non-coarse case, this implies that $$0 = \langle \CWu_{\rho \circ f}^{n}(K), (\sum_{h \in A} h_\ast\Phi) \rangle = \sum_{h \in A} \langle \CWu^n_{\rho \circ f}(K), h_\ast\Phi\rangle $$ $$= \sum_{h \in A}\ \langle \CWu_{\rho \circ f \circ h}^{n}(K),\Phi \rangle = |A|\langle\CWu^{n}_{\rho \circ f}(K),\Phi\rangle \ne 0$$ which is a contradiction. 

For contractible manifolds not homeomorphic to $\rr^n$, we can do a stabilization trick. We assume without loss of generality that $M^n$ is open. If $G$ acts properly on $M^n$, then $\zz \times G$ acts properly on $M^n \times \rr$, which is homeomorphic to $\rr^{n+1}$. Since $\eqobdim(G \times \zz) = \eqobdim(G) + 1$ by the Cone Lemma, this is a contradiction by the above.  
\end{proof}

\begin{lemma}
If $K_1$ and $K_2$ are $H_i$-complexes and $f_i: \Cone_\infty K_i \rightarrow EG_i$ are $H_i$-preserving, then the product map $$f_1 \times f_2: \Cone_\infty(K_1 \ast K_2) = \Cone_\infty K_1 \times \Cone_\infty K_2 \rightarrow EG_1 \times EG_2$$ is $H_1 \times H_2$-preserving. \end{lemma}

\begin{proof}
Let $(h_1, h_2) \in H$. By assumption, there are elements $g_{h_1}$ and $g_{h_2}$ and proper homotopies connecting $f_i \circ h_i$ to $g_{h_1} \circ f_i$. The product of these homotopies gives a homotopy between $f_1 \times f_2 \circ (h_1, h_2)$ and $(g_{h_1}, g_{h_2}) \circ f_1 \times f_2$. Since for each simplex $\gs_i \in K_i$, the image under the homotopy of $\Cone_\infty(\gs)$ is contained in the $R$-neighbourhood of $\Cone_\infty(\gs)$, the same holds true for the image of $\Cone_\infty(\gs_1) \times \Cone_\infty(\gs_2)$ under the product homotopy. 
\end{proof}

Therefore, the join lemma for equivariant obstructors immediately gives the following product formula for $\eqobdim$. 
\begin{Theorem}
$\eqobdim(G_1 \times G_2) \ge \eqobdim(G_1) + \eqobdim(G_2) - 1$
\end{Theorem}

Since the Join Lemma only requires one of the complexes to have a group action, we can also say something about $\eqobdim(G_1 \times G_2)$ when we know $\eqobdim(G_1)$ and $\obdim_{\zz}(G_2)$.

\begin{lemma}
$\eqobdim(G_1 \times G_2) \ge \eqobdim(G_1) + \obdim_\zz(G_2)$
\end{lemma}

If $G$ acts properly and cocompactly on a CAT(0) space $X$, then embedded obstructor complexes $K$ into the boundary $\partial_\infty X$ give proper expanding maps of $\Cone_\infty(K)$ into $X$, and hence give lower bounds for obstructor dimension. Similarly, if $K$ is an obstructor complex in $\partial_\infty X$ which is invariant setwise under the $G$-action on $\partial_\infty(X)$, then this should give lower bounds for equivariant obstructor dimension.

\begin{lemma}\label{l:boundary}
Suppose $G$ acts properly and cocompactly on a CAT(0) space $EG$, and let $\partial_\infty EG$ be the visual boundary for $EG$. Suppose $K$ is an $H$-equivariant $n$-obstructor, and $i: K \rightarrow \partial_\infty(G)$ is an embedding. Suppose that for all $h \in H$, there is $g_h \in G$ so that $i \circ h(K)$ is homotopic to $g_h \circ i(K)$ and the image of each simplex in $K$ is stable under the homotopy. Then $\eqobdim(G) \ge n+3$. 
\end{lemma}

\begin{proof}
Choose a basepoint $x_0 \in EG$ and define $f: \Cone_\infty(K) \rightarrow EG$ by coning $i(K)$ to $x_0$. We claim $f$ is $H$-preserving. By assumption, the maps $f \circ h(\Cone_\infty(hK))$ and $g_h \circ f(\Cone_\infty K))$ have homotopic boundary values. Use this homotopy to homotope $f$ to $f':\Cone_\infty(hK) \rightarrow \rr^n$ which has the same value on the boundary as $g_h \circ F$. This homotopy is proper and by assumption preserves the subspaces $\Cone_\infty(\sigma)$. 

Therefore, $g_h \circ f$ sends $(x,t)$ to a geodesic based at $g_h(x_0)$, and $f'$ sends $(x,t)$ to the asymptotic geodesic based at $x_0$. Since $EG$ is CAT(0), the distance between $g_h \circ F(x,t)$ and $F'(x,t)$ is uniformly bounded by the distance between $g_h(x_0)$ and $x_0$. Since $EG$ is uniformly contractible, we can homotope $g_h \circ f$ to $f'$ and move points a uniformly bounded distance. 
\end{proof}

\section{Examples of Groups}\label{s:examples}

\subsection{Virtual RAAG's}

The simplest examples of groups with $\updim(G) < \actdim(G)$ are virtually free groups that do not act on the plane. We will compute the equivariant obstructor dimension of a more general class of groups which are finite extensions of right-angled Artin groups. We recall the definition.

\begin{definition}
Suppose $L^1$ is a simplicial graph with vertex set $V$.
The \emph{flag complex determined by} $L^1$ is the simplicial complex $L$ whose simplices are the (vertex sets of) complete subgraphs of $L^1$.
Associated to $L^1$ there is a RAAG, $A_L$. 
A set of generators for $A_L$ is $\{g_v\}_{v\in V}$; there are relations $[g_v, g_{v'}] = 1$ whenever $\{v,v'\} \in \edge L^1$.
\end{definition}

Let $T^V$ denote the product $(S^1)^V$.
For each simplex $\gs\in L$, let $T(\gs)$ denote the subtorus $(S^1)^{|\gs|}$. 
The \emph{Salvetti complex} for $A_L$ is the union of the subtori $T(\gs)$ over simplices $\gs$ in $L$:
\[
	S(L):= \bigcup_{\gs\in L} T(\gs).
\]

If a finite group $H$ acts on a flag complex $L$, then $H$ acts on $A_L$ by permuting the generators of $A_L$. Therefore, we can form the semidirect product $A_L \rtimes H$. Suppose now that $L$ is a flag $H$-complex which is an equivariant $(n+1)$-obstructor. 
\begin{Theorem}
If $L$ is a $d$-dimensional flag-complex which is an equivariant $(2d+1)$-obstructor, then $$\eqobdim(A_L) \rtimes H \ge 2d+3.$$
\end{Theorem}

\begin{proof}
Fix a point $\ast$ in the universal cover $\widetilde S(L)$ which is a lift of the unique vertex of $S(L)$. Inside $\widetilde S(L)$, there is a unique lift $\rr_\sigma$ of $T_\gs$ containing $\ast$. For each $T_\gs$, let $\rr_\gs^+$ be the points with nonnegative coordinates. Then the union of the boundaries of the $\rr^+_\gs$ is homeomorphic to $L$. Furthermore, the action of $H$ on $\widetilde S(L)$ permutes these lifts, and stabilizes this copy of $L$ in $\partial_\infty \widetilde S(L)$. Furthermore, the restriction of the action to this copy is precisely the original $H$-action. Since $L$ is an equivariant $n+1$-obstructor, we conclude from Lemma \ref{l:boundary} that $\eqobdim(A_L) \rtimes H = n+3$.
\end{proof}

\begin{remark}
If $L$ is a $d$-dimensional complex, then $\actdim(A_L) = 2d+2 = 2\gdim(A_L)$ \cite{ados}. Since equivariant $(n+1)$-obstructors have $H_n(L; \zz/2) \ne 0$, we are making quite a strong assumption on $L$ (for example, in \cite{ados} $L$ could be a triangulation of an $n$-sphere whereas we require $L$ to be more like a $n$-fold join of $m$ points). 
\end{remark}

\subsection{Products of virtually $3$-manifold groups}
We recall the examples of virtually $3$-manifold groups constructed in \cite{hst} (the examples in \cite{kk} have similar proofs, which we explain in the next subsection). We start with two closed surfaces $S_a$ and $S_b$ of genus $\ge 2$, and a choice of essential simple closed curves $\gamma_a$ and $\gamma_b$ on $S_a$ and $S_b$ respectively. We build a $2$-complex $X_{mn}$ by attaching an annulus to $S_a \sqcup S_b$. We glue one boundary component of the annulus to $\gamma_a$ by a degree $m$-map, and the other boundary component to $\gamma_b$ along a degree $n$-map, see Figure \ref{f:hst}. Let $G_{mn} = \pi_1(X_{mn})$. Hruska, Stark, and Tran show the following:

\begin{Theorem}[Theorem 5.6, \cite{hst}]
For all $m$ and $n$, $G_{mn}$ is virtually a $3$-manifold group. It is a $3$-manifold group if and only if 
\begin{itemize}
\item m  = n = 1
\item m = 1, n = 2 and $\gamma_b$ is non-separating. 
\item m = 2, n = 2, and $\gamma_a$ and $\gamma_b$ are non-separating.
\end{itemize}
\end{Theorem}

In fact, they show $G_{mn}$ virtually embeds as a subgroup of a right-angled Coxeter group $W$ with planar boundary. The Davis complex of $W$ can be $W$-equivariantly thickened to a $3$-manifold. 
\begin{figure}\label{f:hst}
\centering
\includegraphics[scale = .7]{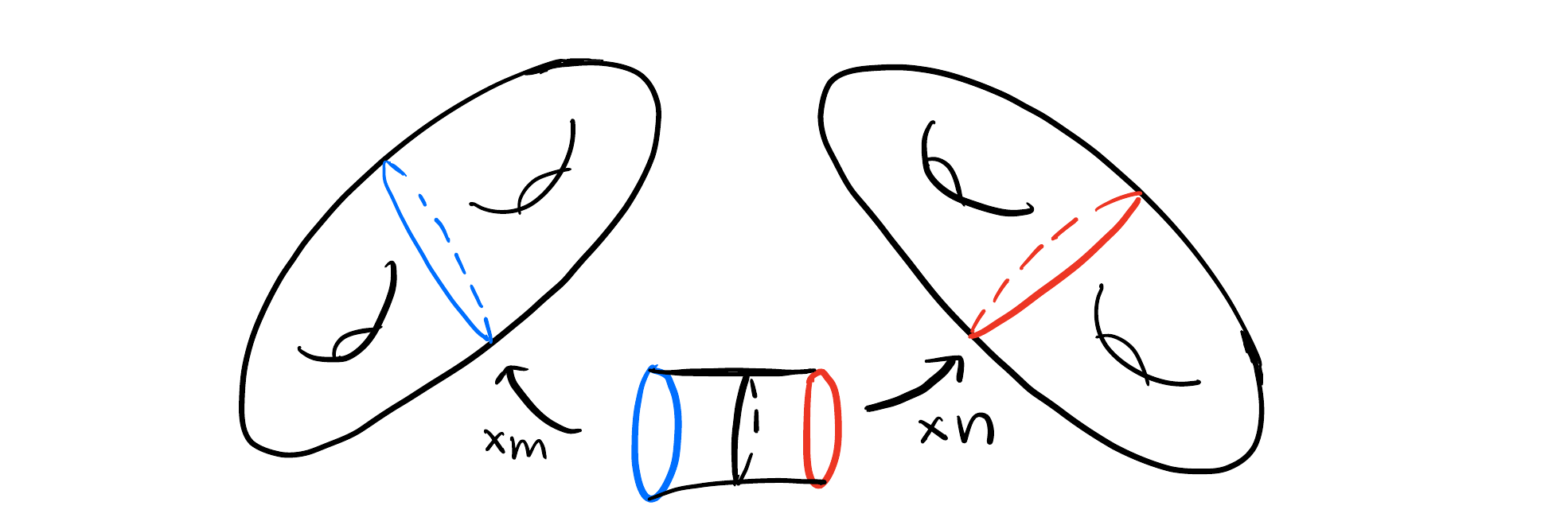}
\caption{The space $X_{mn}$ as in \cite{hst}. Each end of the cylinder is glued onto the corresponding curve with a positive degree map. }
\end{figure}

\begin{Theorem}\label{t:eqhst}
If $m$ or $n$ is divisible by $4$, then $\eqobdim(G_{mn}) = 4$.
\end{Theorem}

\begin{proof}
First assume that $m = 4$ and $n >1$. 
Let $A_{mn}$ denote the quotient space of the cylinder $S^1 \times [0,1]$ under the identifications $(z,0) \cong (e^{2\pi i/m}z, 0)$ and $(z,1) \cong (e^{2\pi i/n}z, 1)$. The universal cover of $A_{mn}$ is $T_{mn} \times \rr$, where $T_{mn}$ is the biregular tree of valence $m$ and $n$. The fundamental group of $A_{mn}$ has presentation $$\pi_1(A_{mn}) = \langle a,b | a^m = b^n \rangle.$$

There is a natural totally geodesic embedding $A_{mn} \rightarrow X_{mn}$. We will identify $\pi_1(A_{mn})$ with its image in $G_{mn}$. 
Inside $EG_{mn} = \widetilde X_{mn}$, choose a copy of $T_{mn} \times \rr$ which the group $\langle a,b \rangle$ acts geometrically on. Let $\widetilde \gamma_a$ denote the axis of $a$ inside this copy of $T_{mn} \times \rr$. Then $\widetilde \gamma_a$ is $v \times \rr$, where $v$ is a valence $m$-vertex in $T_{mn}$. The element $a$ cyclically permutes the $m$-edges emanating from $v$ and translates $n$-units in the $\rr$-direction.  Furthermore, if $P_0$ is a geodesic ray based at $v \in T_{mn}$, then $\langle a \rangle$ cyclically permutes the collection of $m$-half planes $\{a^i(P_0 \times \rr)\}_{i = 0}^m$. The universal cover of $S_a$ is glued to this union of half planes along $\widetilde \gamma_a$; $\langle a \rangle$ acts on this universal cover by a hyperbolic translation. Let $H$ be one of the half planes in $\widetilde S_a$ that $\gamma_a$ bounds.

Let $K = \Cone((m+1) \text{ points})$. 
We define an embedding $f: K \rightarrow \partial_\infty EG_{mn}$. 
We send the cone point to $\gamma_a^{+\infty}$, one of the points to $\gamma_a^{-\infty}$, and the other $m$ points to the endpoints of the $a^iP_0$, see Figure \ref{f:hst1}. We extend this to the cone on $K$ by sending one interval to the boundary of $H$, and the other $m$ intervals to the boundaries of the $a_i(P_0 \times \rr^+)$. We let the group $\zz_m = \langle h \rangle$ act on $\Cone K$ in the usual way by fixing one interval and permuting the other $m$. 

The $\langle a \rangle$-action on $\partial_\infty EG_{mn}$ cyclically permutes the boundaries of $\{a^i(P_0 \times \rr)\}_{i = 0}^m$ and fixes setwise $\partial_\infty \widetilde S_a$. Since hyperbolic translations are isotopic to the identity, there is a homotopy from $a^i \circ f(\Cone K)$ to $f \circ h^i(K)$ which preserves the images of simplices of $K$. Since $\Cone K$ is an equivariant $2$-obstructor, by Lemma \ref{l:boundary},  we have that $\eqobdim(G_{mn}) = 4$. 
\begin{figure}\label{f:hst1}
\centering
\begin{tikzpicture}[scale = .7]

\begin{scope}
   \clip (0, -3) rectangle (3,3);
    \draw[fill = gray, opacity = .3] (0,0) circle(3);
      \draw (0,0) circle(3);
    \draw (0,3) -- (0, -3);
\end{scope}

\begin{scope}
   \clip (7, -3) rectangle (10,3);
    \draw (7,0) circle(3);
      \draw (7,0) circle(3);
   \end{scope}
   
  \draw(7,3) -- (7.75, 0); 
    \draw(7,3) -- (7.25, 0); 
 \draw(7,3) -- (6.75, 0); 
\draw(7,3) -- (6.25, 0);
   \draw[fill = black] (7.75,0) circle(.05);  
      \draw[fill = black] (7.25,0) circle(.05);  
         \draw[fill = black] (6.75,0) circle(.05);  
            \draw[fill = black] (6.25,0) circle(.05);  
               \draw[fill = black] (7,3) circle(.05);  
                  \draw[fill = black] (7,-3) circle(.05);  
   
   \node at (7, -4) {$K$};
   
 \draw[->](5.5,0) -- (3.5,0);   
   \node[above, scale = .8] at (4.5, 0) {$\partial_\infty \widetilde X_{mn}$};
 \draw (0,3) -- (-3,2);
  \draw (0,-3) -- (-3,-4);
   \draw (-3,2) -- (-3,-4);
   
    \draw[->, dashed] (0,0) -- (-2,1);
    \draw[->, dashed] (0,0) -- (-4,0);
       \draw[->] (0,0) -- (-3,-1);
      \draw (0,3) -- (0,-3);
    \draw (0,3) -- (-4,3);
  \draw (0,-3) -- (-4,-3);
   \draw (-4,3) -- (-4,-3);
   \draw (-4,-3) -- (-5,-3.25);
     \draw (-4,-3) -- (-5,-2.75);
 \draw (-3,-4) -- (-3.5,-4.5);
  \draw (-3,-4) -- (-4,-4);
   
     \draw (0,3) -- (-2,3.5);
  \draw (0,-3) -- (-2,-2.5);
   \draw[dashed] (-2,3.5) -- (-2,-2.5);
   
     \draw (-2.5,-2.25) -- (-2,-2.5);
      \draw (-2.75,-2.5) -- (-2,-2.5);
       \draw (0,-3) -- (-2,-2.5);
   
   \node at (1.5, 0) {$\widetilde S_a$};
      \node[scale = .8] at (-6, 0) {$\widetilde A_{mn} \cong T_{mn} \times \rr$};
     \node[above] at (0, 3) {$\gamma_a^{+\infty}$};
   \node[below] at (0, -3) {$\gamma_a^{-\infty}$};

\end{tikzpicture}
\caption{A piece of the universal cover of $X_{mn}$. The complex $T_{mn} \times \rr$ is glued along $\widetilde \gamma_a$ to the universal cover $\widetilde S_a \cong \mathbb{H}^2$. For $m,n \ge 2$ we map $K$ into $\partial_\infty X_{mn}$ by mapping one interval to the boundary of a half space in $\mathbb{H}^2$, and the other $m$ intervals to the boundary of hyperplanes in $\widetilde A_{mn}$.}
\end{figure}
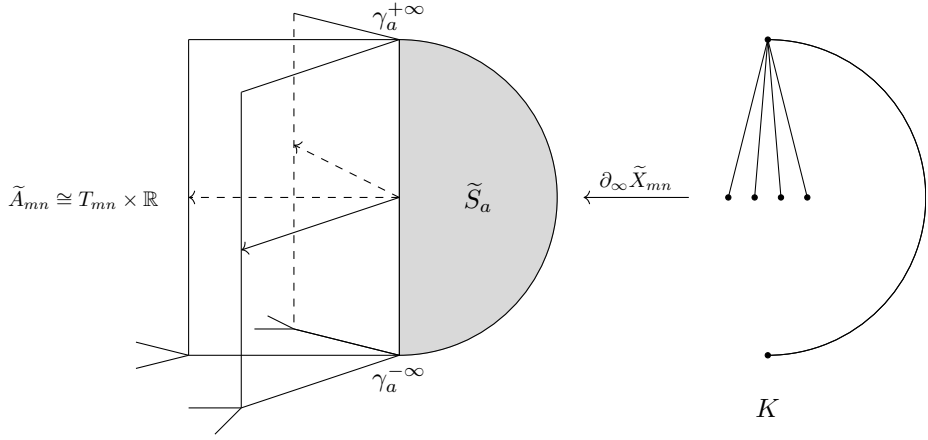

Now, suppose that $n = 1$. The proof in this case is similar; but with a slightly different choice of hyperplanes. In this case, for each edge $e_i$ adjacent to $v$ in $T_{mn}$, $e_i \times \rr$ intersects a lift of $\widetilde S_b$ in $\widetilde X_{mn}$ along a lift of $\gamma_b$. Label these lifts by $\widetilde S_b^i$ and $\gamma_b^i$ respectively. Choose a geodesic $\gamma_i'$ in $\widetilde S_b^i$ which is perpendicular to $\gamma_b^i$ and which intersects $e_i \times \rr$ in $\partial e_i \times 0$. Let $Q_i$ denote one of the quadrants bounded by $\gamma_b^i$ and $\gamma_i'$ that is mapped into itself by a positive translation along $\gamma_b^i$. 

Let $W$ be the union of the $Q_i$ along with $H_a$. The action of $\langle a \rangle$ fixes setwise the hyperplane $H_a$, and simultaneously cyclically permutes the $Q_i$ and acts on them by a hyperbolic translation along $\gamma_b^i$. Define $f: K \rightarrow \partial_\infty W \subset \partial_\infty X$ which again sends the cone point to $\gamma_a^{+\infty}$ and sends each interval to the boundary of $Q_i$. Note that $a^i\partial_\infty Q_i$ is strictly contained in $h^i\partial_\infty Q_{i+1}$. Furthermore, the two embeddings of $K$ into $\partial_\infty W$ are homotopic and the image of each simplex $\gs$ under the homotopy is contained in $f(\sigma)$. Therefore, by Lemma \ref{l:boundary}, $\eqobdim(G_{m1})= 4$. The same argument obviously works for $m$ a multiple of $4$, as we can choose a $\zz/4$-subgroup of $\zz/m$. 
\end{proof}

By the product formula for $\eqobdim$, we see that $$\eqobdim(\prod_k G_{mn}) = 3k+1$$ as long as $m$ or $n$ is divisible by $4$. We can also take the product of $\prod_k G_{mn}$ and any group where we know $\obdim_\zz(G)$. 

\subsection{Kapovich-Kleiner Examples}

We now briefly describe how the same methods work for the examples in \cite{kk}. These are constructed by starting with a hyperbolic surface $S$ with two boundary components $\gamma_a$ and $\gamma_b$, and gluing a cylinder connecting the boundary components via maps of degree $p$ and $q$ respectively. Let $X_{pq}$ denote this space and let $G_{pq} = \pi_1(X_{pq})$.  

Assume that $p = 4$, and choose a lift of $\gamma_a$ in the universal cover $\widetilde X_{pq}$ which $\langle a \rangle$ acts on. A neighborhood of $\widetilde \gamma_a$ in $\widetilde X_{pq}$ is homeomorphic to $K \times \rr$, where $K$ is the cone on $5$ points. Denote by $e_i$ the edges of $K$.  The action of $\langle a \rangle$ fixes say $e_0$ (which is contained in $\widetilde S_a$) and cyclically permutes $e_1, e_2, e_3$ and $e_4$. For $i > 0$, we say such an $e_i$ is a \emph{branch} in $\widetilde X_{pq}$. We want to extend this to an embedding of $\Cone_\infty(K)$ inside $\widetilde X_{pq}$.

To do this, we now construct an $\langle a \rangle$ invariant subcomplex $W$ of $\widetilde X_{pq}$ homeomorphic to $\Cone_\infty K$. We first $\langle a \rangle$-equivariantly choose a branch for each other lift of $\gamma_a$ and $\gamma_b$ in the lift of $S_a$ containing our chosen lift of $\gamma_a$. Next, glue on all lifts of $S_a$ that intersect these branches, as well as the lifts that intersect the four original branches. Continue in this way, choosing for each new lift of $\gamma_a$ and $\gamma_b$ single branches that intersect the new lifts of $S^a$ in lifts of $\gamma_a$ and $\gamma_b$. The resulting complex is homeomorphic to $\Cone_\infty(K) \times \rr$. Furthermore, the action of $\gamma_a$ cyclically permutes four of the halfplanes (while also shifting by a hyperbolic translation). The same argument as in the last example gives us that $\eqobdim(G_{pq}) = 4$.


\begin{bibdiv}
	\begin{biblist}
	
	
	\bib{ados}{article} {
		
		AUTHOR = {Avramidi, Grigori}, 
		author = {Davis, Michael W.}, 
		author = {Okun, Boris},
		author = {Schreve, Kevin}, 
		TITLE = {Action dimension of right-angled Artin groups}, 
		Journal = {Bull. of the London Math. Society},
		Volume = {48}, Year = {2016}, Number = {1}, Pages = {115--126},
		date = {2015},
		
		}
		
		
\bib{bf}{article} {
		AUTHOR = {Bestvina, Mladen}, author = {Feighn, Mark}, TITLE = {Proper actions of lattices on contractible manifolds}, JOURNAL = {Invent.
		Math.},
		
		VOLUME = {150}, YEAR = {2002}, NUMBER = {2}, PAGES = {237--256}, ISSN = {0020-9910},
		
		URL = {http://dx.doi.org.proxy.lib.ohio-state.edu/10.1007/s00222-002-0239-6}, }
		
		\bib{bkk}{article} {author = {Bestvina, Mladen}, author = {Kapovich, Michael}, author = {Kleiner, Bruce}, TITLE = {Van {K}ampen's embedding obstruction for discrete groups}, JOURNAL = {Invent.
		Math.},
		,
		VOLUME = {150}, YEAR = {2002}, NUMBER = {2}, PAGES = {219--235}, ISSN = {0020-9910},
		
		URL = {http://dx.doi.org.proxy.lib.ohio-state.edu/10.1007/s00222-002-0246-7}, } 
		


\bib{bh}{book}{
  author={Bridson, Martin R.},
  author={Haefliger, Andr{\'e}},
  title={Metric spaces of non-positive curvature},
  series={Grundlehren der Mathematischen Wissenschaften [Fundamental
  Principles of Mathematical Sciences]},
  volume={319},
  publisher={Springer-Verlag, Berlin},
  date={1999},
  pages={xxii+643},
  isbn={3-540-64324-9},
  review={\MR{1744486}},
  doi={10.1007/978-3-662-12494-9},
}



\bib{dh}{article}{
  author={Davis, Michael W.},
  author={Huang, Jingyin},
  title= {Determining the action dimension of an Artin group by using its complex of abelian subgroups},
  Journal = {Bull. London Math. Soc}, 
  Year = {2017},
pages = {725-741},
volume = {49},
}




\bib{despotovic}{thesis}{ 
title ={ Action Dimension of Mapping Class Groups}, 
author ={ Despotovic, Zrinka}, 
year = {2006}, 
school ={ Department of Mathematics, University of Utah}, 
type = {phd}, } 

\bib{drutukapovich}{book}{
  author={Drutu, Cornelia},
  author={Kapovich, Misha},
title={Geometric Group Theory},
  series={Colloquium Publications},
  volume={63},
  publisher={American Mathematical Society},
  date={2018},
  pages={xxii+819},




}
 
  \bib{f}{article}{
	author = {Flapan, Erica},
	title = {Symmetries of Mobius Ladders},
	journal = {Mathematische Annalen},
	volume = {283}, YEAR = {1989},
	number = {2},
	pages = {271--283},

 }
 
   \bib{ffn}{article}{
	author = {Flapan, Erica},
	author = {Fletcher, Will},
	author = {Nikkuni, Ryo}
	title = {Reduced Wu and generalized Simon invariants for spatial graphs},
	journal = {Math. Proc. of the Camb. Phil. Soc},
	volume = {156}, YEAR = {2014},
	number = {3},
	pages = {521-544},

 }
 
  \bib{hs}{article}{
  author={Hambleton, Ian},
    author={Savin, Lucian},
journal = {Homology, Homotopy and Applications}
  title={Coarse geometry and P. A. Smith theory},
date = {2011},
pages = {73-102},
volume = {13}, 
number = {2}

}

 \bib{hst}{article}{
  author={Hruska, Chris},
    author={Stark, Emily},
     author={Tran, Hung Cong},
  title={Surface group amalgams that (don't) act on $3$-manifolds},
note = {arXiv:1705.01361},

}

 \bib{vk}{article}{
	author = {van Kampen, E.~R.},
	title = {Komplexe in euklidischen Räumen},
	journal = {Abh. Math. Sem. Univ. Hamburg},
	volume = {9}, YEAR = {1933},
	number = {1},
	pages = {72--78},

 }

 \bib{kk}{article}{
  author={Kapovich, Michael},
    author={Kleiner, Bruce},
  title={Coarse Alexander duality and duality groups},
  journal={J. Differential Geometry},
  volume={69},
  date={2005},
  pages={279-352},

}

 \bib{kirby}{article}{
  author={Kirby, Robion},
   title={Stable homeomorphisms and the Annulus Conjecture},
  journal={Annals of Mathematics},
  volume={89},
  number = {3},
  date={1969},
  pages={575-582},

}


 \bib{melikhov}{article}{

	author = {Melikhov, S. A.},
	title = {The van {K}ampen obstruction and its relatives},
	journal = {Proc. Steklov Inst. Math.},

	volume = {266}, YEAR = {2009},
	number = {1},
	pages = {142\ndash 176},
	issn = {0371-9685},

	url = {http://dx.doi.org.proxy.lib.ohio-state.edu/10.1134/S0081543809030092}, }

  

		




\bib{ks}{article}{
  author={Schreve, Kevin},
  title={Action dimension of lattices in Euclidean buildings},
  journal={Algebr. Geom. Topol.},
  volume={4},
  pages = {3257-3277}
  date={2018},

}	



 \bib{stallings}{unpublished}{
 
	author = {Stallings, J.R.},
	title = {Embedding homotopy types into manifolds},
	date = {1965}, note={unpublished},
	url = {http://math.berkeley.edu/~stall/embkloz.pdf} }

 \bib{taniyama}{article}{
	author = {Taniyama, Kouki},
	title = { Homology classification of spatial embeddings of a graph},
	journal = {Topology and its Applications},
	date = {1995}, 
	pages = {205-288},
	volume = {65}
} 

	\end{biblist}
\end{bibdiv}
\end{document}